\documentclass[11pt,leqno]{article}
\usepackage[english,activeacute]{babel}
\usepackage{epsfig}
\usepackage[latin1]{inputenc}
\usepackage[T1]{fontenc}
\usepackage{amsmath,amsfonts,amssymb,mathrsfs,amsthm}
\usepackage{graphicx}
\usepackage[mathscr]{eucal}
\usepackage{makeidx}
\usepackage{verbatim}
\usepackage{graphics,graphicx}
\usepackage{textcomp}

\usepackage{textcomp}

\usepackage{pstricks}
\usepackage{graphicx}
\usepackage{graphpap}

\usepackage{ifpdf}
\ifpdf
\usepackage{graphicx}  \DeclareGraphicsExtensions{.pdf,.mps}

\else
\usepackage{amsfonts}
\usepackage{graphicx}  \DeclareGraphicsExtensions{.ps,.eps}

\fi

\def\pasdegrille{\let\grille = \pasgrille}
\def\ecriture#1#2{\setbox1=\hbox{#1}
\dimen1= \wd1
\dimen2=\ht1
\dimen3=\dp1
\grille #2 \box1 }
\def\aat#1#2#3{
\divide \dimen1 by 48
\dimen3=\dimen1
\multiply \dimen1 by #1
\advance \dimen1 by -\dimen3
\divide \dimen1 by 101
\multiply \dimen1 by 100
\divide \dimen2 by \count11
\multiply \dimen2 by #2 
\setbox0=\hbox{#3}\ht0=0pt\dp0=0pt
  \rlap{\kern\dimen1 \vbox to0pt{\kern-\dimen2\box0\vss}}\dimen1= \wd1
\dimen2=\ht1}
\def\pasgrille{
\count12= \dimen1 
\divide \count12 by 50
\divide \dimen2 by \count12
\count11 =\dimen2
\ 
\divide \dimen1 by 48
\setlength{\unitlength}{\dimen1}
\smash{\rlap{\ }}
\dimen1= \wd1
\dimen2=\ht1
}
\def\grille{
\count12= \dimen1 
\divide \count12 by 50
\divide \dimen2 by \count12
\count11 =\dimen2
\ 
\divide \dimen1 by 48
\setlength{\unitlength}{\dimen1}
\smash{\rlap{\graphpaper[1](0,0)(50, \count11)}}
\dimen1= \wd1
\dimen2=\ht1
}

\def\e{{\varepsilon}}

\newcommand\R{\mathbb{R}}

\newcommand\N{\mathbb{N}}
\newcommand\Z{\mathbb{Z}}

\newcommand\bna{\begin{eqnarray*}}
\newcommand\ena{\end{eqnarray*}}

\newcommand\bnan{\begin{eqnarray}}
\newcommand\enan{\end{eqnarray}}

\newcommand\bnp{\begin{proof}}
\newcommand\enp{\end{proof}}

\newcommand\bneq{\begin{eqnarray*}\left\lbrace \begin{array}{rcl}}
\newcommand\eneq{\end{array} \right.\end{eqnarray*}}
\newcommand\bneqn{\begin{eqnarray}\left\lbrace \begin{array}{rcl}}
\newcommand\eneqn{\end{array} \right.\end{eqnarray}}

\newcommand\intT{\int_0^T}

\newcommand\nor[2]{\left\|#1\right\|_{#2}}
\newcommand\norL[1]{\left\|#1\right\|_{L^2}}
\newcommand\sgn{\textnormal{sgn}}

\newcommand\Tu{\mathbb{T}^1}

\newcommand{\pc}{ \usefont{T1}{cmtl}{m}{n} \selectfont}

\def\pasdegrille{\let\grille=\pasgrille}
\def\ecriture#1#2{\setbox1=\hbox{#1}
\dimen1= \wd1
\dimen2=\ht1
\dimen3=\dp1
\grille #2 \box1 }
\def\aat#1#2#3{
\divide \dimen1 by 48
\dimen3=\dimen1
\multiply \dimen1 by #1
\advance \dimen1 by -\dimen3
\divide \dimen1 by 101
\multiply \dimen1 by 100
\divide \dimen2 by \count11
\multiply \dimen2 by #2 
\setbox0=\hbox{#3}\ht0=0pt\dp0=0pt
  \rlap{\kern\dimen1 \vbox to0pt{\kern-\dimen2\box0\vss}}\dimen1= \wd1
\dimen2=\ht1}
\def\pasgrille{
\count12= \dimen1 
\divide \count12 by 50
\divide \dimen2 by \count12
\count11 =\dimen2
\ 
\divide \dimen1 by 48
\setlength{\unitlength}{\dimen1}
\smash{\rlap{\ }}
\dimen1= \wd1
\dimen2=\ht1
}
\pasdegrille

\newtheorem{theorem}{Theorem}[section]
\newtheorem{remark}{Remark}[section]
\newtheorem{lemme}{Lemma}[section]

\newtheorem{hypo}{Assumption}[section]

\date{}

\title{Internal control of the Schrödinger equation}
\author{Camille \textsc{Laurent}\footnote{CNRS, UMR 7598, Laboratoire Jacques-Louis
Lions, F-75005, Paris, France  }
\footnote{UPMC Univ Paris 06, UMR 7598, Laboratoire Jacques-Louis Lions,
F-75005, Paris, France, email: { \pc laurent@ann.jussieu.fr }}}

\begin{document}

\maketitle
\begin{abstract}
In this paper, we intend to present some already known results about the internal controllability of the linear and nonlinear Schrödinger equation.  

After presenting the basic properties of the equation, we give a self contained proof of the controllability in dimension $1$ using some propagation results. We then discuss how to obtain some similar results on a compact manifold where the zone of control satisfies the Geometric Control Condition. We also discuss some known results and open questions when this condition is not satisfied.  
Then, we present the links between the controllability and some resolvent estimates. Finally, we discuss the new difficulties when we consider the Nonlinear Schrödinger equation. 
\end{abstract}
\vspace{0.2 cm}{\bf AMS subject classifications.} 93B05, 35Q41, 35Q55

\section{Introduction}
In the control of PDE, the aim is to bring the solution from an initial state to a final state fixed in advance, with a control term which could be, for instance, a source (distributed or internal control), boundary (boundary control) or potential (bilinear control) term, see Lions \cite{Lionscontrolbook} or Coron \cite{Coronlivre} for a general introduction.  

In this paper, is presented some existing results about the internal controllability of the linear and nonlinear Schrödinger equation
\bna
i\partial_t u+\Delta u=f(u)+1_{\omega}(x)g\quad, \quad (t,x)\in [0,T]\times M
\ena  
where $M$ is an open set or a manifold, $g$ is the control and $\omega\subset M$ is an open set. 

The main question will actually be the following: what are the conditions on $\omega$ and $T$ that allow to get the controllability? We expect $\omega$ and $T$ to be the smallest possible. 

It appears that the problem of control will be strongly linked to the propagation of the energy of the solutions and will be therefore strongly linked to the geometry of $\omega$ with respect to $M$. More precisely, a crucial interest will be made to the following condition, called Geometric Control Condition
\begin{center}
Any (generalized) geodesic, meets $\omega$ in a time $t\leq T_0$.
\end{center}
In the presence of boundary (which mainly, will not be considered in these notes), for Dirichlet condition for instance, the generalized geodesics are considered bouncing on the boundary following the laws of geometric optics. This assumption was first considered for the wave equation by Rauch and Taylor for a manifold  \cite{RauchTaylordecay} and by Bardos, Lebeau and Rauch \cite{BLR} for bounded open set and then proved to be sufficient for the Schrödinger equation by Lebeau \cite{control-lin1}. Lebeau \cite{control-lin1} deals with boundary control, but the same ideas lead to the same result for internal control. The present article intends to give an almost complete proof in the more simple case without boundary, in the spirit of Dehman-Gérard-Lebeau \cite{control-nl}. We will also discuss the link with resolvent estimates and the problems posed for the semilinear equation.

\medskip

In the first Section, we will present the main properties of the linear Schrödinger equation and express the controllability in terms of an observability estimate for the free equation, as usual in the HUM method.

Then, in the second Section, we present some results of propagation for some solutions of the Schrödinger equation. These results express the fact that the solutions propagate with infinite speed according to the geodesics of $M$. The two types of information that we are able to propagate are the compactness and the regularity. The final result asserts that a sequence of solutions which is compact (resp. smooth) on an open set $\omega$ satisfying the Geometric Control Condition is compact (resp. smooth) everywhere. We will first present an elementary proof in the one dimensional case and then give the microlocal tools that are necessary to understand the higher dimensions in the boundaryless case.

In Section \ref{sectioncontrol0}, we will show how the previous propagation results allow to prove the controllability under the Geometric Control Condition on a compact manifold. 
We also prove that the HUM control is as smooth as the initial data.

It turns out that the Geometric Control Condition is sufficient but not necessary for the controllability of the linear Schrödinger equation. The necesary and sufficient condition is a widely open problem. We give some known results in that direction. 

In Section \ref{sectionresolv}, we make the link between the controllability (which is equivalent to observability) and some resolvent estimates for the static Laplace operator: 
\bna
\forall \lambda\in \R, \forall u\in D(-\Delta), \quad \norL{u}^2\leq M\norL{(\Delta-\lambda)u}^2+m\norL{1_{\omega}u}^2 
\ena
This point of view not only is interesting for giving alternative proof of the observability but also turns out to be very useful for the link with observablity of the discretized operator, to make the link with other equations (as the wave equation)...

In Section \ref{sectionsemilin}, we present the new problems for the nonlinear equation. We give a sketch of the usual proof of stabilisation, without refering to the functional spaces necessary in that problems.

The Appendix presents some technical steps used in the proofs.
\begin{remark}
This paper represents the notes of a course about the control of the Schrödinger equation given by the author at the summer school PASI-CIPPDE in Santiago de Chile. It was mainly intended to students or young researcher, not specialists of the subject. Therefore, it does not pretend to be a general survey of the latest results in the field, but rather an accessible introduction to the subject. The result presented are mainly not from the author. Moreover, we have sometimes sacrified the optimality or the generality of the results to make the presentation more elementary. In particular, we only deal with manifolds without boundaries, even if most of the results presented here remain true for a manifold with boundary and for control from the boundary.
\end{remark}

\section{The linear equation}
We deal with the internal controllability of the linear Schrödinger equation. If $M$ is a compact manifold or a domain of $\R^d$ and $\omega \subset M$ an open set. The problem is, given $u_0$ and $u_1$ functions on $M$ (in a certain functional space), can we find a control $g$ supported in $[0,T]\times \omega$, such that the solution of 
\bneqn
\label{controleq}
i\partial_t u+\Delta u&=&1_{\omega}g\\
u(0)&=&u_0
\eneqn 
satisfies $u(T)=u_1$?

The problems can be, for instance, formulated for $u_0$, $u_1 \in L^2(M)$ with a control in $L^1([0,T],L^2(M))$ where the Cauchy problem is well posed thanks to semi-group theory, with additional boundary conditions if necessary. 
By linearity and backward wellposedness, the problem can easily be reduced to the case $u_1=0$. Moreover, by duality, the HUM method, which will be described in Section \ref{sectioncontrol0} (see also Lions \cite{Lionscontrolbook} or Tucsnak-Weiss \cite{TucsnakWeiss} for more abstract framework) gives that the controllability in $L^2$ is equivalent to the observability estimate
\bnan
\label{inegobservgener}
\nor{v_0}{L^2}^2 \leq C_T\int_{[0,T]}\nor{1_{\omega}(x)v}{L^2(M)} ^2~dt
\enan
for $v$ solution of the free equation
\bneqn
\label{freeeq}
i\partial_t v+\Delta v&=&0\\
v(0)&=&v_0.
\eneqn
The argument is actually close to the fact that an operator $A$ is onto if its dual satisfies $\nor{x}{}\leq C\nor{A^{*}x}{}$.

Note that the observability is also a property interesting for itself since it quantifies the problem of finding a complete solution from the knowledge of its evolution on a subdomain $\omega$. If the observability holds, there is even an explicit algorithm that allows to find the initial data from the observation (see for instance Ito-Ramdani-Tucsnak \cite{AlgoTimeRevers} using time reversal algorithm). Therefore, the controllability will be closely linked to the study of the propagation of the information for solution of the Schrödinger equation. 

First, the Schrödinger equation is meant to describe the evolution of a quantum particle and the physical heuristic is the following: 

A particle with small pulsation $h$, located in phase space close to a point $x_0$ and a direction $\xi_0$ (take for instance a wave packet $u_0=\frac{1}{h^{d/4}}e^{-\frac{|x-x_0|^2}{2h}}e^{i\frac{x\cdot \xi_0}{h}}$) will travel at speed $1/h$ according to the geodesic (or straight line in the flat case) starting at $x_0$ in direction $\xi_0$, eventually bouncing on the possible boundary. It is clear that this heuristic is not completely true since there are some limitations for the localization in phase space due to the Uncertainty principle. Moreover, the propagation is only true for short times after which some dispersion occurs (of the order of $h$ or better the Ehrenfest time $h |\log(h)|$, see the survey of Anantharaman-Mac\'ia \cite{AnanthaMaciaSurvey} for further comments and references). But this propagation at "infinite" speed gives the idea that the global geometry will be very important.   

\medskip

Note that using the splitting $\partial_t^2 +\Delta^2= (-i\partial_t +\Delta)(i\partial_t+\Delta)$, some controllability results for the Schrödinger equation can easily be transfered to the case of the plate equation (see Lebeau \cite{control-lin1}).

The main purpose of this notes is to give some geometric conditions on $\omega$ that ensure the observability and therefore the controllability. Then, another interesting problem is to find the best $\omega$ (with some appropriate constaints) for which the cost of control is minimal. This problem was investigated in Privat-Trélat-Zuazua \cite{OptimaObservPTZ} (see also the references therein).

\medskip

\textbf{Notation:}

\medskip
The one dimensional torus will be denoted by $\Tu=\R/\Z$ and $L^2(\Tu)$, often denoted by $L^2$, is the space of $L^2$ periodic functions. $M$ will denote a compact manifold without boundary with a Riemannian metric.

The infinitesimal generator of $i\partial_x^2$ (resp. $i\Delta$ where $\Delta$ is the Laplace-Beltrami operator $M$) is denoted by $e^{it\partial_x^2}$ (resp. $e^{it\Delta}$). Hence $u=e^{it\partial_x^2}u_0$ is the solution of 
\bneq
i\partial_t u+\partial_x^2 u&=&	0\\
u(0)&=&u_0.
\eneq
We have $\nor{u(t)}{H^s}=\nor{u_0}{H^s}$ for every $s\in \R$ and $t\in \R$.

Denote $D^r$ the operator defined on $\mathcal{D'}(\Tu)$ by
\bnan
\begin{array}{rclc}
\label{Dr}
\widehat{D^r u}(n)&=& \sgn(n)|n|^r \widehat{u}(n) & \quad \textnormal{if}\quad n\in \Z^*\\
&=& \widehat{u}(0) & \quad \textnormal{if} \quad n= 0.
\end{array}
\enan
where $\widehat{u}$ is the Fourier transform or $u$.

\section{Propagation of compactness}

\subsection{In 1 Dimension}
In this section, we give an elementary proof of some theorems of propagation of compactness in the case of dimension $1$. In the next subsection, will be described the original and more involved proof from Dehman-Gerard-Lebeau \cite{control-nl} in higer dimension using microlocal analysis. The one dimensional version has also been considered in \cite{LaurentNLSdim1} with some Bourgain spaces in the context of nonlinear controllability and stabilization.  

\begin{theorem}
\label{thmpropagation3}
Let $u_n$ be a sequence of smooth solutions of
$$i\partial_t u_n + \partial_x^2 u_n =f_n$$  
with
$$ \left\|u_n\right\|_{L^2([0,T],L^2(\Tu) )} \leq C,~~\left\|u_n\right\|_{L^2([0,T],H^{-1}(\Tu))} \rightarrow 0~~ and~~\left\|f_n\right\|_{L^2([0,T],H^{-1}(\Tu))}\rightarrow  0. $$ 
Moreover, we assume that there is a non empty open set $\omega\subset \Tu$ such that $u_n \rightarrow 0$ in $L^2([0,T],L^2(\omega))$.\\
Then $u_n \rightarrow 0$ in $L^2_{loc}([0,T],L^2(\Tu))$.
\end{theorem}
\begin{proof} 
Let us consider the real valued functions $\varphi \in C^{\infty}(\Tu)$ and $\Psi \in  C^{\infty}_0(]0,T[)$, which will be chosen later. Set $Bu=\varphi(x)D^{-1}$ and $A=\Psi(t)B$ where $D^{-1}$ is the operator defined in (\ref{Dr}). We have $A^*=\Psi(t)D^{-1}\varphi(x)$.\\
Denote $L$ the Schrödinger operator $L  = i\partial_t  + \partial_x^2 $. We write by a classical way
\begin{eqnarray*}
\alpha_{n}&=& (L u_n,A^*u_n)_{L^2(]0,T[\times \Tu)}-(Au_n,Lu_n,)_{L^2(]0,T[\times \Tu)}\\
&=& ([A,\partial_x^2]u_n,u_n)_{L^2(]0,T[\times \Tu)}-i(\Psi'(t)Bu_n,u_n)_{L^2(]0,T[\times \Tu)}.
\end{eqnarray*}
We have also
\begin{eqnarray*}
\alpha_{n}&=& (f_n,A^*u_n)_{L^2(]0,T[\times \Tu)}-(Au_n,f_n)_{L^2(]0,T[\times \Tu)}.
\end{eqnarray*}
We obtain
\begin{eqnarray}
\label{estimfn}
\left|(f_n,A^*u_n)_{L^2(]0,T[\times \Tu)}\right| &\leq& \|f_n\|_{L^2([0,T],H^{-1})} \|A^*u_n\|_{L^2([0,T],H^{1})} \nonumber\\
&\leq& \|f_n\|_{L^2([0,T],H^{-1})} \|u_n\|_{L^2([0,T],L^2)}.
\end{eqnarray}
Then, $\left|(f_n,A^*u_n)_{L^2(]0,T[\times \Tu)}\right|\rightarrow 0$ when $n \rightarrow \infty$. The same estimate holds for the other term and gives $ \alpha_{n} \rightarrow 0$. Similarly, the term $ (\Psi'(t)Bu_n,u_n)_{L^2(]0,T[\times \Tu)}$ converges to zero.\\
Finally, we get
  $$([A,\partial_x^2]u_n,u_n)_{L^2(]0,T[\times \Tu)} \rightarrow 0 \textnormal{ when } n \rightarrow \infty. $$
Since $D^{-1}$ commutes with $\partial_x^2$, we have
\bna
[A,\partial_x^2]=-2\Psi(t)(\partial_x \varphi) \partial_x D^{-1}- \Psi(t)(\partial^2_x \varphi)D^{-1}.
\ena
Making the same estimates as in (\ref{estimfn}), we get  $$(\Psi(t)(\partial^2_x \varphi)D^{-1}u_n,u_n)_{L^2(]0,T[\times \Tu)} \rightarrow 0.$$ 
Moreover, $-i \partial_x D^{-1}$ is actually the orthogonal projection on the subspace of functions with $\widehat{u}(0)=0$. Using weak convergence, we easily obtain that $\widehat{u_n}(0)(t)$ tends to $0$ in $L^2([0,T])$ and indeed, $$(\Psi(t)(\partial_x \varphi)\widehat{u_n}(0)(t),u_n)_{L^2(]0,T[\times \Tu)} \rightarrow 0.$$
The final result is that for any $\varphi\in C^{\infty}(\Tu)$ and $\Psi \in  C^{\infty}_0(]0,T[)$
$$(\Psi(t)(\partial_x \varphi)u_n,u_n)_{L^2(]0,T[\times \Tu)} \rightarrow 0.$$
Notice that the functions which can be written $\partial_x \varphi$ are indeed all the functions $\psi$ that fulfill $\int_{\Tu} \psi =0$. For example, take any $f \in C^{\infty}_0(\omega)$ and any $x_0\in \Tu$, then $\psi(x)= f(x)-f(x-x_0)$ can be written by $\psi=\partial_x \varphi$. The strong convergence in $L^2([0,T],L^2(\omega))$ implies  
 $$(\Psi(t)f u_n,u_n)_{L^2(]0,T[\times \Tu)}\rightarrow 0.$$
Then for any $x_0\in \Tu$
$$(\Psi(t)f(.-x_0) u_n,u_n)_{L^2(]0,T[\times \Tu)}\rightarrow 0.$$
We close the proof by constructing a partition of the unity of $\Tu$ with functions with support smaller than $\omega$. \end{proof}
\begin{remark}
The previous theorem allows a source term $f_n$ bounded in a lower order Sobolev norm (only $L^2H^{-1}$ while $u_n$ is bounded in $L^2L^2$). This fact can be extremely useful in a nonlinear context, where the source term comes from the nonlinearity.
\end{remark}
A closely related result that can be useful in other situations is the propagation of the regularity. For some solution of the Schrödinger equation with source term, we recover some regularity information from an open set $\omega$ to the whole circle.
We write Proposition 13 of \cite{control-nl} in the one dimensional setting, which can be obtained with a proof very similar to Theorem \ref{thmpropagation3}.
\begin{theorem}
\label{thmpropagreg}
Let $T>0$ and $u\in L^2([0,T],H^r(\Tu))$, $r \in \R$, solution of
$$i \partial_t u +\partial_x^2 u=f \in L^2([0,T],H^r) $$
Additionally, we assume that there exist an open set $\omega$ and $\rho \leq \frac{1}{2}$ such that  \mbox{$u\in L^2_{loc}(]0,T[,H^{r+\rho}(\omega))$}.\\
Then, we have $u\in L^2_{loc}(]0,T[,H^{r+\rho}(\Tu))$.
\end{theorem}
This kind of result can be very useful if $f$ is either zero or a nonlinear term of $u$. This allows to iterate the result to get the smoothness of some solution $u$ smooth on $[0,T]\times \omega$.
\subsection{The higher dimensions}
The proof of Theorem \ref{thmpropagation3} contains all the ingredients of the following theorem which is due to Dehman-Gérard-Lebeau \cite{control-nl}, with the assumption of Geometric Control Condition:
\begin{hypo}[Geometric Control Condition]
\label{hypoGCC}
We say that an open set $\omega \subset M$ satisfies the Geometric Control Condition if there exists $T_0\geq 0 $ such that any geodesic with velocity one issued at $t=0$ meets $\omega$ in a time $0\leq t\leq T_0$.
\end{hypo}
\begin{theorem}
\label{thmpropagM}[Dehman-Gérard-Lebeau]
Let $M$ be a compact manifold and $\omega\subset M$ satisfying the Geometric Control Condition. Then, the same result of Theorem \ref{thmpropagation3} holds.
\end{theorem}
The proof is quite similar to the one we have presented for Theorem \ref{thmpropagation3} except that it requires the use of microlocal analysis. The propagation happens in the phase space $(x,\xi)\in T^*M$ and not only in space as in dimension $1$. The link with the geometry is made using pseudodifferential operators (see Alinhac-Gérard \cite{AlinhacGerard} for an introduction). 

For a symbol $f(x,\xi)$ on the phase-space $T^*M$, we will denote with big letter $F(x,D)$ one pseudodifferential operator with principal symbol $f(x,\xi)$. We will mainly use the three following facts
\begin{enumerate}
\item \label{pseudoSob}Any pseudodifferential operator of order $r$ sends $H^s(M)$ into $H^{s-r}(M)$.
\item \label{pseudoCrochet} For $a_1(x,D)$ and $a_2(x,D)$ two pseudodifferential operators of respective orders $r_1$ and $r_2$, the commutator $[a_1(x,D), a_2(x,D)]$ is a pseudodifferential operator of order $r_1+r_2-1$ of principal symbol $\frac{1}{i}\left\{a_1,a_2\right\}=\frac{1}{i}H_{a_1}a_2$ where $\left\{\cdot,\cdot\right\}$ is the Poisson bracket and $H_{a_1}$ the Hamiltonian field of $a_1$. The Hamiltonian of a function $a_1(x,\xi)$ on the cotangent bundle $T^*M$ can be expressed in coordinates as $H_{a_1}=\sum_i \frac{\partial a_1}{\partial \xi_i}\frac{\partial }{\partial x_i}-\frac{\partial a_1}{\partial x_i }\frac{\partial }{\partial \xi_i}$.
\item The principal symbol of the Laplace-Beltrami operator $p(x,D)=-\Delta$ is $p(x,\xi)=|\xi|_x^2$ where $|\cdot|_x^2$ is the metric on $T^*M$ inherited from the Riemannian metric.
\end{enumerate}
The following modifications have to be made to the 1D proof to get the same result in higher dimension:

\begin{itemize}
 \item we pick $B$ as a pseudodifferential operator of order $-1$ on $M$ and $A=\Psi(t)B$. However, $B$ has the same effect on Sobolev spaces thanks to fact \ref{pseudoSob}, and therefore, the same estimates will remain true.
 \item From the same computation, we obtain 
 \bnan
 \label{propagident}
 \left(\Psi(t)[B,\Delta]u_n,u_n\right)_{L^2}\rightarrow0.
 \enan
 \item The symbol of $[B,\Delta]$ is $\frac{1}{i}H_pb$ where $H_p$ is the Hamiltonian of the symbol $p(x,\xi)=|\xi|_x^2$ thanks to fact \ref{pseudoCrochet}. 
\end{itemize}
If $p$ is the symbol corresponding to the norm of the flat metric $p=|\xi|^2=\sum_i \xi_i^2$, then the Hamiltonian trajectory starting from a point $(x_0,\xi_0)$ is the straight line (in $x$) $(x(t),\xi(t))=(x_0+2t\xi_0,\xi_0)$. In the more general case, $p=|\xi|_x^2$ where $|\cdot|_x$ is the metric on $T^*M$ inherited from the Riemannian metric. The Hamiltonian flow describes the geodesic flow and $x(t)$ is a geodesic, up to renormalisation.

This allows to prove that there is a propagation of the information along the flow of $H_{|\xi|_x^2}$, that is along the geodesic flow of the manifold. There are several ways (essentially equivalent) to prove this propagation of compactness. The first, that we will sketch, is to mimick what we did for dimension $1$. The propagation can be made step by step in the phase-space. The final result of the computation we made (namely (\ref{propagident})) is that for any pseudodifferential operator of order $0$, whose principal symbol can be written $q=\frac{1}{i}H_pb$ with $b$ of order $-1$, we have   $\left(\Psi(t)Q(x,D)u_n,u_n\right)_{L^2}\rightarrow0$ as $n\rightarrow +\infty$. To apply this in dimension $1$, we just noticed that a class of symbols that can be written $q=H_pb$ is the one of the form $q=f(x)-f(x-x_0)$. If $f$ is supported in a neighborhood of a point $y_0$, it allows to get information around $y_1=y_0+x_0$ from information around $y_0$. The equivalent of this fact in higher dimension is presented in the following geometric lemma, which, this time, is in the phase-space $T^*M$:
\begin{lemme}
\label{lmpropageom}
Let $\rho_0 \in T^*M\setminus 0$, $\Gamma(t)$ be the bicaracteristic  starting at $\rho_0$ for the symbol $p(x,\xi)=|\xi|_x^2$ (i.e. the solution of $ \dot{\Gamma}(t)=H_p(\Gamma(t))$, $\Gamma(0)=\rho_0$). Then, there exists $\varepsilon>0$ such that if $0<t<\varepsilon$, $\rho_1=\Gamma(t)$, and $V_1$ a small conical neighborhood of $\rho_1$, there exists a neighborhood $V_0$ of $\rho_0$ such that for any symbol $c(x,\xi)$ homogeneous of order $0$, supported in $V_0$, there exists another symbol $b(x,\xi)$ homogeneous of order $-1$ such that  
\bnan
\label{identityporpag}
H_p b(x,\xi)=c(x,\xi)+r(x,\xi)
\enan
where $r$ is of order $0$ and supported in $V_1$. 
\end{lemme}
We give a sketch of the proof in the appendix, which mostly relies on solving a transport equation. Note that this propagation can easily be globalized (that is without the $\varepsilon$) by using a compactness argument and by iterating the process of propagation.

We can then use this lemma to propagate the information from a neighborhood of $\rho_1$ to a neighborhood of $\rho_0$.

If $\rho_1\in \omega$, the strong convergence of $u_n$ on $\omega$ implies 
$$\left(\Psi(t)R(x,D)u_n,u_n\right)_{L^2}\rightarrow 0 $$
Moreover, the previous computation gave $\left(\Psi(t)[B,\Delta]u_n,u_n\right)_{L^2}\rightarrow0$. Using identity (\ref{identityporpag}) for the symbols, this gives finally $$\left(\Psi(t)C(x,D)u_n,u_n\right)_{L^2}\rightarrow 0 $$ where $c(x,\xi)$ can be chosen equal to $1$ around $\rho_0$. We have proved the strong convergence in a microlocal region close to $\rho_0$. 

 The Geometric Control Condition states that any point $\rho_0 \in T^*M$ can be linked with a bicaracteristic to another point in $T^*\omega$. Then, by iterating the previous result, we get that for any $\rho_0 \in T^*M$, we can find $c(x,\xi)$ equal to $1$ around $\rho_0$ such that $\left(\Psi(t)C(x,D)u_n,u_n\right)_{L^2}\rightarrow 0 $. By performing a partition of the unity of $T^*M$, we can conclude the proof of Theorem \ref{thmpropagM} as in dimension $1$.

This propagation can also be expressed from the point of view of microlocal defect measure introduced by P. Gérard \cite{defectmeasure} and L. Tartar \cite{tartarh-measure}. For a sequence $u_n$ weakly convergent to $0$ in $L^2$, an associated microlocal defect measure is a measure on $]0,T[\times S^*M$ (where $S^*M=\left\{(x,\xi)\in T^*M\left| |\xi|_x=1\right.\right\}$) so that 
\bna
\left(A(t,x,D_x)u_n,u_n\right)_{L^2([0,T]\times M)}\rightarrow \int_{S^*M}a_0(t,x,\xi)~d\mu(t,x,\xi) 
\ena
for any tangential (i.e. not depending on the variable dual to the time $t$) pseudodifferential operator $A(t,x,D_x)$ of order $0$ with principal symbol $a_0$.
This measure represents the locus of concentration of the energy of $u_n$ is the phase space. Therefore,
\begin{itemize}
\item $u_n \rightarrow 0$ in $L^2([0,T],L^2(\omega))$ can be expressed as $\mu\equiv 0$ on $]0,T[\times S^*\omega$.
\item $\left([A,\Delta]u_n,u_n\right)_{L^2}\rightarrow0$ for any $A$ of order $-1$ is equivalent to $H_p \mu=0$, which can be interpretated as $\mu$ being invariant by the geodesic flow.
\end{itemize} 
Therefore, the assumption of Geometric Control Condition gives that $\mu\equiv 0$ and the strong convergence of $u_n$ to $0$ in $L^2([0,T],M)$. 

\section{Linear controllability}
\label{sectioncontrol0}
\subsection{In 1 Dimension}
\begin{theorem}
\label{thmobservlin}
Let $\omega$ be a non empty open set of $\Tu$ and $T>0$. Then, there exists $C>0$ such that
\bnan
\label{inegobserv}
\left\|u_0 \right\|^2_{L^2} \leq C_T \intT \left\| 1_{\omega}(x)e^{it\partial_x^2}u_0 \right\|^2_{L^2} ~dt
\enan
for every $u_0\in L^2(\Tu)$.
\end{theorem}
\begin{remark}
There are several methods to prove the previous theorem: microlocal analysis \cite{control-lin1,control-nl}, multiplier method \cite{Machtyngier}, Carleman estimates \cite{LasieckaTrigH1}, moment theory, etc. We are going to present one that is very close to some microlocal ideas but without the technicity of microlocal analysis (because we are in dimension $1$). It contains the ideas that give the result under Geometric Control Condition in higher dimension. Another intersest is that it is quite robust to perturbation and allows to deal with nonlinear problems as is done in \cite{control-nl} \cite{LaurentNLSdim1,LaurentNLSdim3}.
\end{remark}
\bnp
Initially, we will consider smooth functions. The result can be easily extended by density. 

We first prove the weaker estimate.
\bnan
\label{inegobservweak}
\left\|u_0 \right\|^2_{L^2} \leq C_T \intT \left\| 1_{\omega}(x) e^{it\partial_x^2}u_0 \right\|^2_{L^2} ~dt + C_T\left\|u_0 \right\|^2_{H^{-2}}
\enan
We are going to apply the strategy of compactness uniqueness, for which we argue by contradiction. 

Let $u_{n,0}$ be a sequence of smooth function on $\Tu$ contradicting (\ref{inegobservweak}) and $u_n:= e^{it\partial_x^2}u_{n,0}$ be the associated linear solution, so that 
\bnan
\label{absurdinegobservweak}
\intT \left\| 1_{\omega}(x)e^{it\partial_x^2}u_{0,n }\right\|^2_{L^2} ~dt + \left\|u_{0,n} \right\|^2_{H^{-2}}\leq \frac{1}{n} \left\|u_{0,n} \right\|^2_{L^2} .
\enan

Since the problem is linear, we can suppose that $\nor{u_{n,0}}{L^2}=1$ (otherwise replace $u_{n,0}$ by  $u_{n,0}/\nor{u_{n,0}}{L^2}$). Estimate (\ref{absurdinegobservweak}) implies $u_{0,n} \rightarrow 0$ in $H^{-2}$. Interpolating between $L^2$ and $H^{-2}$, we get the same convergence in $H^{-1}$. Therefore, $u_n \rightarrow 0$ in $L^2([0,T],H^{-1}(\Tu))$ by conservation of the norm $H^{-1}$. Moreover, (\ref{absurdinegobservweak}) gives also $u_n \rightarrow 0$ in $L^2([0,T],L^2(\omega))$. So, we are in position to apply Theorem \ref{thmpropagation3} with $f_n=0$ and obtain $u_n \rightarrow 0$ in $L^2_{loc}([0,T],L^2(\Tu))$. This is in contradiction with $\nor{u_{n,0}}{L^2}=1$ and proves (\ref{inegobservweak}).

Now, we get back to the proof of (\ref{inegobserv}).

Denote $N_T=\left\{u_0\in L^2 \left|e^{it\partial_x^2}u_0=0 \textnormal{ on } ]0,T[\times\omega\right.\right\}$. $N_T$ is a linear subspace of $L^2(\Tu)$. We want to prove that $N_T=\{0\}$. Let $\e>0$ and $u_0\in N_{T}$. For $\e$ small enough such that  $u_{\e}=\frac{e^{i\e\partial_x^2}u_0-u_0}{\e}$ is in $N_{T/2}$. Since $u_0\in L^2$, $u_{\e}$ is bounded in $H^{-2}$ uniformly in $\e\rightarrow 0$. Estimate (\ref{inegobservweak}) gives $\left\|u_{\e} \right\|^2_{L^2} \leq  C_{T/2}\left\|u_{\e} \right\|^2_{H^{-2}}$. By definition of $e^{it\partial_x^2}$, $\partial_x^2 u_0\in L^2$ and $u_0\in H^2$. This immediatly gives $\partial_x^2 u_0\in N_T$. Therefore, $N_T$ is stable by $\partial_x^2$ and only contains smooth functions. Applying again estimate (\ref{inegobservweak}) to $\partial_x^2 u_0$ when $u_0\in N_T$, we get $\nor{\partial_x^2 u_0}{L^2} \leq C_T \nor{u_0}{L^2}$. Indeed, the unit ball of $N_T$ (for the $L^2$ topology) is bounded in $H^2$ and therefore is compact by the Rellich theorem. So, by the theorem of Riesz, $N_T$ is finite dimensional. Therefore, the operator $\partial_x^2$ sends $N_T$ into itself and admits an eigenvalue $\lambda$ associated to $u_{\lambda}$. Moreover,  $\partial_x^2 u_{\lambda}=\lambda u_{\lambda}$ and $u_{\lambda}=0$ on $\omega$ implies $u_{\lambda}=0$ (by the Cauchy-Lipschitz theorem). This implies $N_T=\{0\}$.

To conclude the proof of Theorem \ref{thmobservlin}, we argue again by contradiction. Let $u_{n,0}$ be a sequence of smooth functions on $\Tu$ of $L^2$ norm $1$ contradicting (\ref{inegobserv}) and $u_n= e^{it\partial_x^2}u_{n,0}$ the associated linear solution, hence
\bnan
\label{absurdinegobserv}
\intT \left\| 1_{\omega}(x)e^{it\partial_x^2}u_{0,n }\right\|^2_{L^2} ~dt \leq \frac{1}{n} \left\|u_{0,n} \right\|^2_{L^2} \leq  \frac{1}{n}.
\enan 
Let $u=e^{it\partial_x^2} u_0$ be a weak limit of $u_n$ (we easily get that $u$ is also solution the Schrödinger equation and $u_0$ is a weak limit of $u_{0,n }$). (\ref{absurdinegobserv}) implies that $u_0\in N_T$ and so $u_0=0$. So $u_{0,n }\rightharpoonup 0$ weakly in $L^2$ and by Rellich theorem, $\nor{u_{0,n }}{H^{-2}}\rightarrow 0$. Now applying our weak estimate (\ref{inegobservweak}) and (\ref{absurdinegobserv}) give $\left\|u_{0,n} \right\|_{L^2}\rightarrow 0$ which is a contradiction.
\enp
Applying the HUM method of J-L. Lions, the previous Theorem implies the exact controllability in $L^2$ of the linear equation. More precisely, we can construct an isomorphism of control $S$ from $L^2$ to $L^2$: for every state $\Psi_0$ in $L^2$, there exists $ \Phi_0=S^{-1}\Psi_0$ ($ \Psi_0=S\Phi_0$) such that if $\Phi $ is solution of the dual equation
 \begin{eqnarray}
\label{eqlinphi}
\left\lbrace
\begin{array}{rcl}
i\partial_t \Phi + \partial_x^2 \Phi &=& 0\\
\Phi(x,0)&=&\Phi_{0}(x)
\end{array}
\right.
\end{eqnarray}
 and $\Psi$ solution of 
\begin{eqnarray}
\label {eqlin}
\left\lbrace
\begin{array}{rcl}
i\partial_t \Psi + \partial_x^2 \Psi &=&1_{\omega} \Phi\\
\Psi(T)&=&0
\end{array}
\right.
\end{eqnarray}
we have $\Psi(0)=\Psi_0$.

Actually, multiplying (\ref{eqlin}) by $\overline{\Phi}$, integrating on $[0,T]\times M$ and integrating by parts, we get
\bna
-i\left\langle \Psi(0),\Phi_0\right\rangle_{L^2(M)}=\int_0^T \int_M |1_\omega \Phi|^2 dxdt 
\ena
Denoting $S\Phi_0= -i\Psi(0)$, this gives 
\bna
\left\langle S\Phi_0,\Phi_0\right\rangle_{L^2(M)}=\int_0^T \int_M |1_\omega \Phi|^2 dxdt 
\ena
This automatically gives that $S$ is nonnegative self adjoint and the observability estimates proves that it is positive and an isomorphism of $L^2$. The controllability follows directly.
\begin{theorem}
\label{thmcontrollin}
Let $\omega$ be a non empty open set of $\Tu$ and $T>0$. Then, for every $u_0$, $u_1\in L^2(\Tu)$ there exists a control $g\in L^{\infty}([0,T],L^2)$ supported in $[0,T]\times \omega$ such that the solution of 
\bneqn
\label{eqncontrol}
i\partial_t u+\Delta u&=&g\\
u(0)&=&u_0
\eneqn 
satisfies $u(T)=u_1$
\end{theorem}

The previous theorem gives the existence of a control in $L^2$ which drives a $L^2$ data to $0$. A natural question is whether this control is smoother if the data is smoother. This turns out to be true if we replace $1_{\omega}$ by a smooth function $\chi_{\omega}$ that satisfies $\chi_{\omega}(x)>C>0$ for $x\in \omega$. This problem was first adressed for the wave equation in Dehman-Lebeau \cite{HUMDehLeb}. The argument can be easily adapted to the Schrödinger case as it is done by the author in \cite{LaurentNLSdim1} and even to a more general framework in Ervedoza-Zuazua \cite{ErvedozaZuazuasmoothHUM}. 

Now, we denote $S$ the operator defined by $S\Phi_0 =-i\Psi(0)$ where $\Psi$ is defined by (\ref{eqlinphi}) and (\ref{eqlin}) with replacing $1_{\omega}$ by $\chi_{\omega}^2$ (which is actually $BB^t$ where $B=\chi_{\omega}$, following the functional analytic framework of HUM). 
 
\begin{theorem}
\label{ThmregHUM}
Let $S$ be defined as before with $\chi_{\omega}$ smooth. Then, $S$ is an isomorphism of $H^s$ for every $s\geq 0$.
\end{theorem}
\begin{proof} We easily see that $S$ maps $H^s$ into itself. So we just have to prove that $S\Phi_0 \in H^s$ implies $\Phi_0 \in H^s$. We use the formula 
$$S \Phi_0 = \int_0^T e^{-it\partial_x^2}\chi_{\omega}^2e^{it\partial_x^2} \Phi_0~dt.$$
Since $S^{-1}$ is continuous from $L^2$ into itself, we get, using Lemma \ref{lemmecommut} of the Appendix, 
\bna
\left\|D^s \Phi_0\right\|_{L^2} & \leq& C\left\|S D^s \Phi_0\right\|_{L^2}\leq  C\left\|\int_0^T e^{-it\partial_x^2}\chi_{\omega}^2e^{it\partial_x^2} D^s \Phi_0\right\|_{L^2}\\
&\leq & C\left\|D^s \int_0^T e^{-it\partial_x^2}\chi_{\omega}^2e^{it\partial_x^2}\Phi_0\right\|_{L^2}\\
&&+C\left\|\int_0^T e^{-it\partial_x^2}\left[\chi_{\omega}^2,D^s\right]e^{it\partial_x^2}\Phi_0\right\|_{L^2}\\
&\leq & C\left\| S \Phi_0 \right\|_{H^s}+ C_s\left\|\Phi_0\right\|_{H^{s-1}}.
\ena
This yields us the desired result for $s\in [0,1]$. We extend this result to every $s\geq 0$ by iteration.\\
\end{proof}

\subsection{Some comments about the higher dimensions}
Applying a similar proof together with the propagation Theorem \ref{thmpropagM} allows to prove the controllability on a compact manifold under the Geometric Control Condition. The only difference in the argument is that in the proof of $N_T=\left\{0\right\}$, we have to replace the Cauchy-Lipschitz theorem by a unique continuation argument for elliptic equations $\Delta u_{\lambda}=\lambda u_{\lambda}$, which is true for any $\omega \neq \emptyset$ (see \cite{Zuilybook} for instance). More precisely, we get the following controllability result:

\begin{theorem}
\label{thmcontrollinM}
Let $\omega \subset M$ where $M$ is a compact manifold and $\omega$ satisfies the Geometric Control Condition. Then, the same result as Theorem \ref{thmcontrollin} (controllability in $L^2$) holds.
\end{theorem}

 This result of controllability was first proved by G. Lebeau \cite{control-lin1} in the more complicated case of a domain with boundary.  He proved the boundary controllability of the Schrödinger equation under the Geometric Control Condition similar to Assumption \ref{hypoGCC}, where the geodesics have to be replaced by generalised rays of geometric optics bouncing on the boundary. Note also that some other results assume stronger geometric assumptions than Geometric Control Condition but less regularity on the coefficients and may then give more explicit results using for instance multiplier techniques (see for instance Zuazua \cite{ZuazSchrodLions}, Machtyngier \cite{Machtyngier} or Lasiecka-Triggiani \cite{LasiTrigSchrod}) or Carleman estimates (see for instance Lasiecka-Triggiani-Zhang \cite{LasieckaTrigH1}).
  
The Geometric Control Condition is known to be necessary and sufficient for the controllability of the wave equation since the results of Bardos-Lebeau-Rauch \cite{BLR} (note that the necessary part can bring subtelties if we choose $1_{\omega}$ instead of a smooth function $\chi_{\omega}$, see \cite{LebeauControlHyp} for the example of the control from half of the sphere). Actually, for any geodesic, it is possible to construct a sequence of solutions of the wave equation that concentrates on this ray. (see Ralston \cite{RalstonWave} with a geometric optic solution or Burq-Gérard \cite{BurqGerardCNS} using microlocal defect measure). For the Schrödinger equation, this kind of construction is, in general, only possible for some small times of the order of $h_n$ (or better the Ehrenfest time $h_n |\log(h_n)|$) for some data oscillating at scale $h_n\rightarrow 0$, which is not sufficient to contradict observability. 
  
  However, for some specific stable trajectories, the previous kind of construction is sometime possible. For instance, Ralston \cite{Ralston} constructed some approximate eigenfunctions (which can of course be translated into solutions of the Schrödinger equation) that concentrate on a stable periodic trajectory satisfying additional assumptions. By stable, we mean that the application of first return of Poincaré has only some eigenvalues of modulus $1$ which are distincts. This stability occurs in particular in the case of positive curvature. This yields us to conclude that any $\omega$ has to meet this specific trajectory if we want to get observability, making the Geometric Control Condition necessary for this trajectory. This stability assumption is quite natural because we expect that a stable periodic trajectory will prevent the dispersion. Indeed, a high frequency particle traveling on such trajectory might remain concentrated close to the trajectory.
  
  For instance, on the sphere $S^2$, the geodesics are the equators, which are stable. There exist explicit quasimodes that concentrate on any such meridian. Namely, writing $S^2\subset  \R^3_{(x,y,z)}$, we consider the sequence of normalised eigenfunctions $u_n=c_n(x+iy)^n$, i.e. $\nor{u_n}{L^2(S^2)}=1$ with $c_n\approx \sqrt{n}$ (see \cite{GallotHulinLaf} Section 4.E.3 for a description of the eigenfunctions of the sphere as the restriction of harmonic homogeneous polynomials of $\R^3$). Since $|u_n|^2=c_n^2(x^2+y^2)^{n}=c_n^2(1-z^2)^n$, this sequence concentrates exponentially on the equator $\left\{z=0\right\}$, contradicting the observability if $\omega$ does not intersect the equator. The Geometric Control Condition is therefore necessary and sufficient on the sphere. Note that in the case of the ball with Dirichlet boundary conditions for instance, there exist also explicit eigenfunctions that are known to concentrate on the boundary so that the control region $\omega$ has to touch the boundary (see Lagnese \cite{Lagnese} Lemma 3.1 where it is stated for the wave equation).
  
Yet, if there are some unstable trajectory, the situation may be more complicated because of dispersion. 

The result of Jaffard \cite{Jaffard}, extended by Komornik \cite{Komornik} to higher dimensions, proves, for the torus $\mathbb{T}^d$, that any non empty open set is enough to obtain the controllability in any time $T>0$ (see also the previous result of Haraux \cite{haraux} when $\omega$ is a strip in a rectangle and the article of Tenenbaum-Tucsnak \cite{Tenenbaum} where they obtain estimates for the boundary control problem on rectangles). Yet, it can happen that such open set does not satisfy the Geometric Control Condition. Note also that a different proof of the same result was given by Burq and Zworski \cite{Burq} and allows to give other examples of domains where Geometric Control Condition is not necessary, as the Bunimovitch stadium. We give a very elementary and general proof of Burq (see \cite{Burq}) that a strip is sufficient for the torus $\mathbb{T}^2$, that is the result of Haraux \cite{haraux}. 
\begin{theorem}[Burq]Let $M_1$, $M_2$ be two compact manifolds (possibly with boundary). Let $\omega_1\subset M_1$ that satisfies an observability estimate like (\ref{inegobservgener}) in time $T$.\\
Then, the same result as Theorem \ref{thmcontrollin} holds true (i.e. controllability in $L^2$) for $\omega=\omega_1\times M_2$.
\end{theorem}    
\bnp
For fixed $x_1$, we decompose $u_0$ according to the eigenfunctions $\varphi_k$ (with their respective eigenvalues $\lambda_k$) of $\Delta_2$ on $M_2$, where $c_k$ are functions in $L^2(M_1)$
\bna
u_0(x_1,x_2)=\sum_{k}c_k(x_1)\varphi_k(x_2).
\ena
This gives 
\bna
\left[e^{it\Delta}u_0\right](x_1,x_2)=\sum_k \left[e^{it\Delta_1}c_k\right](x_1) e^{it\lambda_k} \varphi_k(x_2).
\ena
By Plancherel formula, the respective $L^2$ norms can be written
\bna
\nor{u_0}{L^2(M_1\times M_2)}^2=\sum_k\nor{c_k}{L^2(M_1)}^2\\
\nor{e^{it\Delta}u_0}{L^2(\omega_1\times M_2)}^2=\sum_k\nor{e^{it\Delta_1}c_k}{L^2(\omega_1)}^2
\ena
For each $k\in \N$, the observability estimate on $\omega_1$ gives  
\bna
\nor{c_k}{L^2(M_1)}^2\leq C_T\int_0^T\nor{e^{it\Delta_1}c_k}{L^2(\omega_1)}^2.
\ena
By summing up, we get the expected estimate
\bna
\nor{u_0}{L^2(M_1\times M_2)}^2\leq C_T \int_0^T\nor{e^{it\Delta}u_0}{L^2(\omega_1\times M_2)}^2.
\ena
\enp
Note also that the previous theorem can be used as a first step to give another proof of the result of Jaffard and Komornik that any open subset is enough for observability on $\mathbb{T}^d$ (see Burq-Zworski \cite{BurqZworskiJEDP}). A more precise description is also given by Mac\'ia \cite{MaciaT2} for dimension $2$ and Anantharaman-Mac\'ia \cite{AnanthaMacitore} for higher dimension, using 2-microlocal defect measures.

Therefore, it seems very temptative to make the conjecture that the good condition is that $\omega$ meets all stable trajectories. However, quite surprisingly, Colin de Verdière and Parisse \cite{ColinParisse} managed to construct some sequence of eigenfunctions concentrating logarithmically on an unstable periodic trajectory of a specific negative curved surface with boundary. 

Therefore, it seems that the global dynamic of the geodesic flow has to be taken into account. Determining the set of trajectories that can be missed by the control zone is very complicated and implies a good understanding of the global dynamic of the geodesic flow.

Since eigenfunctions are particular solutions of Schrödinger equation, the observability is strongly linked to the spreading of high energy eigenfunctions. Some progress has been recently made in this subject (in particular, in relation with the Quantum Unique Ergodicity conjecture) and it could certainly be transfered to controllability problems. This is done for instance in the work of Anantharaman and Rivière \cite{AnanRiv} in negative curvature using entropy properties. We also refer to the recent survey of Anantharaman and Mac\'ia \cite{AnanthaMaciaSurvey}.

It seems also that to permit some possible loss of derivative (that is to control only some smoother data with a less regular control) can be a natural setting in some geometries where "few" trajectories miss $\omega$. For instance, the article of Burq \cite{Burqcontrolobstacle} proves the controllability in an open set with a finite number of convex holes assuming some further assumptions. The boundary control is supported in the exterior boundary (of the big open set). Therefore, there exist some trapped trajectories going back and forth between the holes, which are very unstable. There is controllability but with loss of $\varepsilon$ derivatives, that is we can control data in $H^{\varepsilon}$ while the regularity of the control produce a priori data in $L^2$. These kind of results are very consistent with some resolvent estimates which give a loss of $\log(\lambda)$ at frequency $\lambda$ when the trapped set is a hyperbolic trajectory (see Christianson \cite{Christiansonhyperbolic}) or a "very small set" (see Nonnenmacher-Zworski \cite{NonnZworskiActa}).

\medskip

Finally, note that if we only aim the approximate controllability, we only need to prove unique continuation result for the free Schrödinger equation. It is actually the fact that an operator has a dense image if its dual is injective. More precisely, the approximate controllability is equivalent to answer:

\medskip
\begin{center}
\emph{
Let $u$ solution of $i\partial_t u+\Delta u=0$, does $u\equiv 0$ on $]0,T[\times \omega$ imply $u\equiv 0$ ?
}
\end{center}
\medskip

This happens to be true for any non empty open set $\omega$. It easily follows from Holmgren theorem for analytic metric or from more complicated unique continuation results in more general metric, see for instance Robbiano-Zuily \cite{UniqZuilyRob}. But the problem is that the approximate controllability does not give any information on the cost of the control to get close to the target. However, without geometric assumption on $\omega$, it is sometime possible to quantify the cost of this approximate control to get at distance $\e$. It appears that the cost explodes exponentially with $1/\e$, as can be deduced by duality from Phung \cite{Phung} Theorem 3.1.

\section{Links with some resolvent estimates}
\label{sectionresolv}
In the previous subsection, we have obtained the controllability of the linear Schrödinger equation from an observability estimate. It turns out that this observability estimate is equivalent to some resolvent estimates. 

In particular, we have the following theorem proved by Miller \cite{Miller}, following ideas of Burq and Zworski \cite{Burq}:
\begin{theorem}[Burq-Zworski, Miller]
\label{thmResolObs}
The system (\ref{eqncontrol}) is exactly controllable in $L^2$ in finite time if and only if there exist $M>0$, $m>0$ so that 
\bnan
\label{resolv}
\forall \lambda\in \R, \forall u\in D(-\Delta), \quad \norL{u}^2\leq M\norL{(\Delta-\lambda)u}^2+m\norL{1_{\omega}u}^2. 
\enan
\end{theorem}
Actually, we can give an estimate on the time $T$ and cost of control $C_T$ of the observability estimate with respect to $M$ and $m$. Although, there is not a complete equivalence: the constants $M$ and $m$ can be written depending on the time and cost of control, but these two expressions are not inverse one of the other.

 Actually, the proof is very general and can be put in an abstract setting for self-adjoint operator. We follow \cite{Miller}.
\bnp
We prove the observability estimate (\ref{inegobservgener}) for a time $T>0$ (that will depend on $M$), since it is equivalent to controllability.
Let $\chi\in C^1_0(\R)$ to be specified later, $u_0\in D(-\Delta)$ and $u(t)=e^{it\Delta}u_0$. $v(t)=\chi(t)u(t)$ is solution of $i\partial_tv+\Delta v=i\dot{\chi}(t)u(t):=f(t)$. The Fourier transform of $f$ is $\hat{f}(\tau)=(-\tau+\Delta)\hat{v}(\tau)$. We apply the resolvent estimate (\ref{resolv}) to $\hat{f}(\tau)$ with $\tau=\lambda$ and get
\bna
\norL{\hat{v}(\tau)}^2\leq M\norL{\hat{f}(\tau)}^2+m \norL{1_{\omega}\hat{v}(\tau)}^2.
\ena 
After integration in $\tau$, the Plancherel formula gives 
\bna
\int_{\R}\norL{v(t)}^2dt\leq M\int_{\R}\norL{f(t)}^2dt+m \int_{\R}\norL{1_{\omega}v(t)}^2dt.
\ena 
Recalling the expression of $v$ and $f$ gives
\bna
\int_{\R}(\chi(t)^2-M\dot{\chi}(t)^2)\norL{u(t)}^2dt\leq m \int_{\R}\chi(t)^2\norL{1_{\omega}u(t)}^2dt.
\ena
Now, we specify $\chi(t)=\phi(t/T)$ for $\phi\in C^{\infty}(]0,1[)$ not zero.
By the conservation of $L^2$ norm, we get 
\bna
I_T\norL{u_0}^2\leq m \nor{\phi}{L^\infty}\int_0^T\norL{1_{\omega}u(t)}^2dt
\ena
with $I_T=\int_0^T(\chi(t)^2-M\dot{\chi}(t)^2)=\frac{\norL{\phi}^2}{T}\left(T^2-M\frac{\norL{\dot{\phi}}^2}{\norL{\phi}^2}\right)$. This expression can be made positive for $T$ large enough. Actually, by optimizing $\phi\in C^{\infty}_0(]0,1[)$, i.e. $\min_{\phi\in C^{\infty}_0(]0,1[)}\norL{\dot{\phi}}/\norL{\phi}=\pi$ (which is obtained with some sequence converging to $\sin(\pi t)$), it is possible to obtain the controllability for a time $T>\pi\sqrt{M}$ with a constant $C_T=2mT/(T^2-M\pi^2)$ in (\ref{inegobservgener}) (see Theorem 5.1 of \cite{Miller}).

Now, let us prove the converse result: observability implies resolvent estimate.\\
Denote $v(t)=\left(e^{it\Delta}-e^{it\lambda }\right)u_0$ solution of 
\bna
\dot{v}(t)=i\Delta e^{it\Delta}u_0-i\lambda e^{it\lambda }u_0=i(\Delta-\lambda)e^{it\Delta}u_0+i\lambda (e^{it\Delta}-e^{it\lambda })u_0:=e^{it\Delta}f+i\lambda v(t)
\ena
where $f=i(\Delta-\lambda)u_0$. So, since $v(0)=0$, we have $v(t)=\int_0^t e^{i(t-s)\lambda}e^{is\Delta} f~ds$ and therefore, $\norL{1_{\omega}v(t)}\leq \norL{v(t)}\leq t\norL{f}$.

We apply the observability estimate to $u_0$
\bna
\norL{u_0}^2&\leq &C_T\int_0^T \norL{1_\omega e^{it\Delta}u_0}^2\leq 2C_T\int_0^T \norL{1_\omega v(t)}^2+2C_T\int_0^T \norL{1_\omega e^{it\lambda}u_0}^2\\
& \leq &C_T \frac{2T^3}{3}\norL{(\Delta-\lambda)u_0}^2+2C_T T\norL{1_{\omega}u_0}^2,
\ena
this gives the result with $M=2C_T T^3/3$ and $m=2C_T T$.
\enp
\begin{remark}
\label{rmkreslarge}
Note that estimate (\ref{resolv}) could, in principle, never be used directly to prove controllability in arbitrary short time, since it would require $M$ to be arbitrary small with a large $m$ eventually. This is not possible because it would for instance imply $\norL{u}\leq \e \norL{\Delta u}$ for any $u\in C^{\infty}_0(\omega^c)$ and $\e>0$. However, it is possible sometimes to prove (\ref{resolv}) with $M$ arbitrary small, but for some $\lambda$, $|\lambda|\geq R_0$, depending on $M$. The following property can be used:

\medskip

Assuming (\ref{resolv}) holds for $|\lambda|\geq R_0$ and $1_{\omega}\varphi=0$ implies $\varphi=0$ for any $\varphi$ eigenvalue of $\Delta$, then the Schrödinger equation is controllable in time $T>\pi\sqrt{ M}$.

\medskip

This is Property 6.6.4 from the book of Tucsnak-Weiss \cite{TucsnakWeiss}. It is obtained by showing that, for $v$ spectrally localised at high frequency (depending on $M$ and $R_0$), the resolvent estimate (\ref{resolv}) is automatically true for $|\lambda|\leq R_0$ by basic spectral inequalities. Since the assumption gives it for $|\lambda|\leq R_0$, we get the resolvent estimate for any $\lambda\in \R$ and for $v$ localised at high frequency. This gives controllability for data at high frequency. Since the controllability is true for the remaining finite dimensional subspace of data spectrally localised at low frequency, the global controllability can be obtained by a theorem of simultaneous controllability (Theorem 6.4.2 of \cite{TucsnakWeiss}).  The second assumption of uniqueness for eigenfunctions is always true for any $\omega \neq \emptyset$ by unique continuation for elliptic operators of order $2$, but we have chosen to give it in an abstract setting.  
\end{remark}
The point of view of controllability through resolvent estimates can be very useful for various reasons:
\begin{itemize}
\item Their proof can be easier than direct observability. For instance, if we use microlocal arguments, we can use some measures which do not depend on time, and which are semiclassic. Moreover, in \cite{Burq}, the authors developed a strategy which allows to use existing resolvent estimates as a black-box to get others which could be useful in other situation. Roughly speaking, if locally the geometric situation is the same as in another geometric setting where you know some resolvent estimates, you can use them as a black-box.
 
\item To make proofs of controllability by the resolvent can be easier to make the link between the observability of the contiuous system and the observability of a discretised system coming from numerical analysis. This approach was used first by Ervedoza-Zheng-Zuazua \cite{ErvZhenZuaz}, see also Miller \cite{MillerResolapprox} for later improvements and references.

\item It can give informations about the cost of the control when the time of control goes to zero, especially when we only aim at controlling the "high frequency" part of the function. In that case, some resolvent estimates are only needed for large $\lambda$, see Miller \cite{Miller}.

\item It can make some links between controllability of different equations. For example, it gives a very simple proof that the controllability of the wave equation implies the controllability of the Schrödinger equation in arbitrary small time. Indeed, the proof of Theorem \ref{thmResolObs} is very general and can be applied to any self-adjoint operator $A$ and control operator $B$ (bounded or admissible see \cite{Miller} or \cite{RamdaTaka}). In particular, if we apply it to the wave operator $A=\left(\begin{array}{cc}
0& Id\\ \Delta & 0 \end{array}\right)$ with a control operator $B=\left(\begin{array}{c}
0\\ 1_{\omega} \end{array}\right)$, we see that the controllability of the wave equation in $H^1\times L^2$ is equivalent to the following resolvent estimate
\bnan
\label{resolvwaveA}
&\forall \lambda\in \R, \forall (u_0,u_1)\in D(A),& \nonumber\\
& \nor{u_0}{H^1}^2+\norL{u_1}^2\leq M_2\left(\nor{u_1-\lambda u_0}{H^1}^2+\norL{\Delta u_0-\lambda u_1}^2\right)+m_2\norL{1_{\omega}u_1}^2. &
\enan
By taking $u_1=\lambda u_0=\lambda u$, we get a resolvent estimate for $\Delta$
\bnan
\label{resolvwave}
\forall \lambda\in \R, \forall u\in D(-\Delta), \quad \norL{\lambda	u}^2\leq M_2\norL{(\Delta-\lambda^2)u}^2+m_2\norL{\lambda1_{\omega}u}^2. 
\enan
We immediatly get that $(\ref{resolvwave})$ implies the resolvent estimate (\ref{resolv}) with $M$ arbitrary small and $\lambda>R_0$ with $R_0=R_0(M)$ large enough. But since $-\Delta$ is positive, the same result is also true for $\lambda<-R_0$ with $R_0$ large enough, by basic spectral theory estimates. This gives (\ref{resolv}) uniformly for $|\lambda|\geq R_0$ with a fixed small $M$ and eventually large $m$. By using Remark \ref{rmkreslarge} and unique continuation for eigenfunctions of $\Delta$ (which is actually a consequence of (\ref{resolvwave}) for $\lambda \neq 0$), we get the general case. Indeed we have proved that the controllability of the wave equation implies (\ref{resolv}) and therefore the controllability of the Schrödinger equation. Note that this implication can also be proved by the so-called "transmutation method" (see Phung \cite{Phung} and Miller \cite{Miller, Millerviolent}) which writes a solution to the Schrödinger equation as an integral kernel using the solution of the wave equation. This method seems more precise to estimate the cost of controllability when $T$ is small. 

Therefore, Theorem \ref{thmcontrollinM} can be deduced directly from the related result of Bardos-Lebeau-Rauch \cite{BLR} for the wave equation.

Note also that it is possible to prove the equivalence between (\ref{resolvwave}) and (\ref{resolvwaveA}), see Yamamoto-Zhou \cite{YamamZhouHautus}, Ramdani-Takahashi-Tenenbaum-Tucsnak \cite{RamdaTaka} by using the equivalence to observability of "wave-packets" or Miller \cite{MillerResolapprox} with a link with the resolvent estimates for $\sqrt{-\Delta}$.

\medskip
 
Note also that, quite surprisingly, Duyckaerts-Miller \cite{DuyckMillerHeat} showed that the controllability of the Schrödinger equation does not necessarily imply the controllability of the heat equation, even if it is the case in many geometric situations.

\end{itemize}

\section{The semilinear equation}
\label{sectionsemilin}
In this part, we aim at giving a short overview of the techniques and problems for the control and stabilization of semilinear Schrödinger equations. 
\subsection{The Nonlinear Schrödinger Equation}
In this section, we will discuss semilinear Schrödinger equations of the form 
$$i\partial_t u+\Delta u =f(|u|^2)u.$$
They can arise in various physical problems as Bose-Einstein condentate, propagation of wave enveloppe in nonlinear optic or non linear propagation (as tsunami) etc. The global well-posedness of this equation is a complicated topic and depends considerably on the nonlinearity $f$ and on the spatial domain where the equation is considered. We refer the reader to the expository books of Cazenave \cite{bookCazenave} or Tao \cite{Tao} which are good introductions to the functional spaces that are usually used for these problems, namely the Strichartz and Bourgain spaces. Note also that there is a strong relation (still not completely understood) between the dispersive properties of the free Schrödinger equation and the geometry of the geodesic flow. 
 
If $f$ is a real valued function, two quantities are formally conserved by the equation:
\bna
&\nor{u}{L^2},& \textnormal{ the mass, }\\
E(u)&=\int |\nabla u|^2~dx+\int F(|u|^2)~dx,&\textnormal{ the $H^1$ energy, }
\ena
where $F$ is a primitive of $f$.

These quantities will be crucial to prove some stabilisation results. Note that the sign of $F$ will be important in the case where the nonlinear equation can only be solved in $H^1$.  
\subsection{General strategy}
\subsubsection{Local controllability}
Concerning the internal controllability, some local results can, in general, be obtained from the linear result by a perturbation argument. The idea is mainly using a fixed point argument in the functional setting inherited from the wellposedness result. This was first done for the nonlinear wave equation by Zuazua \cite{zuazua} and in Dehman-Lebeau-Zuazua \cite{DLZstabNLW} for stronger nonlinearities. This is also done for Nonlinear Schrödinger equations, using Strichartz estimates in Gérard-Dehman-Lebeau \cite{control-nl} or Bourgain spaces in Rosier-Zhang \cite{RosierZhang} and Laurent \cite{LaurentNLSdim1,LaurentNLSdim3}. 

More precisely, for the control to $0$, the strategy is the following. We want to find a solution of
\bneqn
\label{controleqnonlin}
i\partial_t u+\Delta u&=&f(|u|^2)u+1_{\omega}g\\
u(0)&=&u_0
\eneqn 
satisfying $u(T)=0$.

We will look for $g$ of the form $1_{\omega}\Phi$ where $\Phi=e^{it\Delta }\Phi_0$ is solution of (\ref{eqlinphi}). We split $u=v+\Psi$ where $v$ contains the nonlinear part and $\Psi$ contains the control:
\begin{eqnarray}
\label {eqlinbis}
\left\lbrace
\begin{array}{rcl}
i\partial_t \Psi + \Delta \Psi &=&1_{\omega} \Phi\\
\Psi(T)&=&0
\end{array}
\right.
\end{eqnarray}
and so
\begin{eqnarray}
\label {eqnonlincontrol}
\left\lbrace
\begin{array}{rcl}
i\partial_t v +  \Delta v &=&f(|u|^2)u\\
v(T)&=&0
\end{array}
\right.
\end{eqnarray}
We notice that all the system only depends on $\Phi_0$. We denote $L\Phi_0=u(0)=\Psi(0)+v(0)=S\Phi_0+K\Phi_0$, where $S$ is again the linear inversible HUM operator defined in (\ref{eqlinphi}) and (\ref{eqlin}) and $K$ is a nonlinear operator. We are looking for $\Phi_0$ such that $L\Phi_0=u_0$, that is $\Phi_0=S^{-1}u_0-S^{-1}K\Phi_0:=B\Phi_0$. So, the main task is to find a fixed point for $B$ by showing that it is contracting on a sufficiently small ball. This can be achieved if $u_0$ is small and using many boot strap arguments showing that if $u_0$ and $\Phi_0$ are small, the solutions $\Psi$, $u$, and $v$ will remain small in the functional space adapted to the nonlinearity. The difficulties come mainly from the nonlinear estimates that  are required.

Note that the control from the boundary is for the moment quite less studied in the nonlinear framework, mainly because the Cauchy problem for nonhomogeneous boundary conditions is less understood. We can cite the work of Rosier-Zhang \cite{RosierZhangRectangle} on rectangles and also \cite{RosierZhang_boundary} by the same authors which takes advantage of the dispersion in $\R^d$ for a control through all the boundary.

\subsubsection{Global controllability}

Obtaining the controllability for large data is in general much more subtle.

First, we can expect to get the result for arbitrary short time if the nonlinearity is not too large, as globally Lipschitz (or log-Lipschitz as in \cite{FernandZuazheat} for the nonlinear heat equation). This strategy was quite well described in the review article of Zuazua \cite{ZuazuasurveySchrod}. Up to the knowledge of the author, it is still not proved.  

If the nonlinearity is not globally Lipschitz, and if $\omega$ is not the whole space, there is no available result of control in arbitrary small time, unlike the linear case. The most common strategy is the one by stabilization and local control. It was applied by Dehman-Lebeau-Zuazua \cite{control-nl} for compact surfaces using Strichartz estimates and by the author \cite{LaurentNLSdim1,LaurentNLSdim3} in some contexts where Bourgain spaces are needed, as in dimension $3$ and in dimension $1$ at the $L^2$ regularity. The idea is to find a good stabilizing term to bring the solution close to zero. During that time, we take as a control the stabilization term given by the stabilized equation. By combining the previous construction with a local controllability near zero, we obtain the global controllability to zero for large data. Additionaly, we notice that the backward equation $i\partial_t u-\Delta u=-f(|u|^2)u$ fulfills exactly the same conditions for controlling to zero. The same reasonning as before allows to get control to zero for this backward equation. By reversing the time, it gives a control to get from zero to our expected final state. Combining these both results gives the global controllability in large time. This strategy is illustrated in Figure \ref{fig.strategy} where the term energy is either the $L^2$ norm or the $H^1$ energy. 
\begin{figure}[!h]
$$\ecriture{\includegraphics[width=0.7\textwidth]{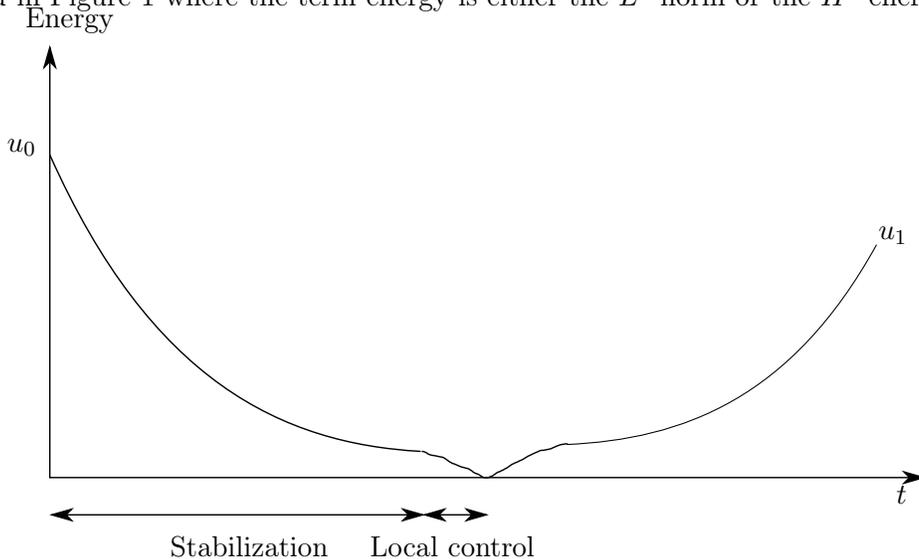}}
{\aat{0}{28}{Energy}\aat{47}{16}{$u_1$}\aat{-1}{21}{$u_0$}\aat{8}{-2}{Stabilization}\aat{19}{-2}{Local control}\aat{48}{1}{$t$}}$$ \caption{Global strategy by stabilization}\label{fig.strategy}
\end{figure}

The difficulty is to prove that the stabilization is indeed effective. To be more concrete, we consider the example of the $1$-dimensional torus treated by the author in \cite{LaurentNLSdim1}, where it is possible to solve the cubic nonlinear equation in $L^2$. The aim is to stabilize the equation in $L^2$. Actually, the fact to consider $L^2$ solutions requires using Bourgain spaces (see \cite{Bourgain}), which are functional spaces specially designed to contain the dispersive properties of the Schrödinger operator. Since this is not the topic of this survey, we will not precise the functional spaces, but all the existence and propagation (of compactness and regularity) have to be stated in these spaces.
 
A natural damping for the $L^2$ norm leads to the following system
\bneq
i\partial_t u+\Delta u +i\chi_{\omega}(x)^2 u&=& \pm |u|^2u\\
u(0)&=&u_0\in L^2.
\eneq
where $\chi_{\omega}$ is a smooth function supported in $\omega$.

So, we have the decay estimate
\bna
\nor{u(T)}{L^2}^2=\nor{u_0}{L^2}^2-2\int_0^T \nor{\chi_{\omega}u(t)}{L^2}^2dt.
\ena
To obtain an exponential decay of the type $\nor{u(t)}{L^2}\leq Ce^{-\gamma t}$, it is sufficient to prove the observability estimate 
\bnan
\left\|u_0 \right\|^2_{L^2} \leq C \intT \left\| \chi_{\omega}(x)u \right\|^2_{L^2} dt
\enan
for some bounded $u_0$. This means that at each step $[0,T]$, a certain proportion of the energy is "burnt". 

A possible proof for such result is the compactness-uniqueness argument similar to the one performed in the linear case as in the proof of Theorem \ref{thmobservlin}. 
We argue by contradiction and assume that there exists a bounded sequence of solutions satisfying
\bnan
 \intT \left\| \chi_{\omega}(x)u_n \right\|^2_{L^2} dt\leq \frac{1}{n}\left\|u_{n,0} \right\|^2_{L^2}.
\enan
This time, since the equation is nonlinear, we have to distinguish two cases (up to a subsequence): Let $\alpha_n =\left\|u_{n,0} \right\|_{L^2}\rightarrow \alpha $, then $\alpha>0$ or $\alpha=0$.

First case $\alpha>0$:
\begin{itemize}
\item We denote $u$ a weak limit of $u_n$. We apply a propagation of compactness, similar to Theorem \ref{thmpropagation3}, to prove that the convergence is actually strong in $L^2_{loc}([0,T],L^2)$. An additional difficulty is arised by the fact that the source term $f_n= \pm |u_n|^2u_n$ lies in some complicated functional spaces as the Bourgain spaces, which imposes a modification in the functional framework of Theorem \ref{thmpropagation3}. Then, the strong convergence in $L^2_{loc}([0,T],L^2)$ can be easily transfered to $C([0,T],L^2)$ using that the flow map is Lipschitz. 

Note that in other situations (this does not happen in $L^2$ for dimension $1$), the sequence of solutions can be proved to be linearizable: some compactness results allow to prove that the source term is strongly convergent in a suitable space which shows that $u_n-v_n$ is strongly convergent in the natural functional space, where $v_n$ is solution of the linear equation for the initial data. Then, it is possible to apply propagation of compactness for the linear equation.
\item The argument of uniqueness using the linear space $N_T$ does not work since the equation is nonlinear. So, it remains to prove the following unique continuation result:

\emph{The only solution in $L^2$ of
\bneqn
\label{eqnuniq}
i\partial_t u+\Delta u&=&\pm |u|^2u\\
u&=& 0 \textnormal{ on } [0,T]\times \omega
\eneqn
is $u\equiv 0$.
}

A first step for showing such result is the propagation of regularity similar to Theorem \ref{thmpropagreg} (but in the setting of Bourgain spaces) which would allow us to reduce the unique continuation result to smooth functions. But even for smooth functions, this unique continuation result possesses some big difficulties in high dimensions. They are described in Subsection \ref{subsectuniq}. 

\item The strong convergence of $u_n$ to $u=0$ in $C([0,T],L^2)$ contradicts the fact that $\alpha>0$. 
\end{itemize}
Second case $\alpha=0$:

If we do the change of unknown $w_n=u_n/\alpha_n$ which satisfies $\norL{w_n(0)}=1$, this should solve 
\bna
i\partial_t w_n+\Delta w_n+i\chi_{\omega}^2w_n=\pm \alpha_n^2|w_n|^2w_n.
\ena
A boot strap argument allows to take advantage of the smallness of $\alpha_n^2$ in front of the nonlinearity to conclude that $\alpha_n^2|w_n|^2w_n$ converges to $0$ in the Bourgain space and the solution is almost linear. The solution $w_n$ still satisfies the local convergence on $\omega$ 
\bna
 \intT \left\| \chi_{\omega}(x)w_n \right\|^2_{L^2} dt\leq \frac{1}{n}.
\ena
We can then conclude easily as in the linear case and get a contradition to $\norL{w_n(0)}=1$.

\medskip

In some more complicated geometries, see \cite{control-nl} \cite{LaurentNLSdim3}, this strategy allows to prove some global stabilization and controllability results under the two conditions 
\begin{enumerate}
\item $\omega$ fulfills the Geometric Control Condition.
\item $\omega$ fulfills some unique continuation property similar to (\ref{eqnuniq}). We refer to subsection \ref{subsectuniq} for some further comments.\end{enumerate}

\bigskip

Another strategy that was proposed by the author \cite{LaurentNLSdim3} is by successive controls close to trajectories.
The idea is still to find some control that make the solution tends to zero, but this time, we use successive controls near free trajectories. We prove that there exists a fixed $\e$ such that for any free trajectory leading $\tilde{u}_0$ to $\tilde{u}_1$, we can control $\tilde{u}_0$ to a final state $u_f$ with $\nor{u_f-\tilde{u}_1}{E}\leq \varepsilon$, where $E$ is an "energy". Since the energy is conserved for each free trajectory, we can choose $u_f$ so that $\nor{u_f}{E}\leq \nor{\tilde{u}_0}{E}-\varepsilon$. The energy is then decreasing at each step and we obtain a control to $0$. By (almost) reversibility, we can do the same process to go from zero to the expected final state. This strategy is illustrated on figure \ref{fig.strategy_successifs}. We have simplified a little the exposition because the conserved nonlinear $H^1$ energy is not exactly the $H^1$ norm for the nonlinear Schrödiner equation. In this scheme of proof, the difficulty is to show the local controllability near free trajectories. Moreover, if we want the strategy to work, we need an uniform $\e$ for all the trajectories in a ball of $H^1$, that is with a weak regularity. In particular, we need to get some observability estimates uniform in the norm of the potentials $V_1$ and $V_2$ for solutions $u$ of
\bna
i\partial_t u+\Delta u+V_1u+V_2\overline{u}=0.
\ena
Since $V_1$ and $V_2$ come from the linearization of an arbitrary bounded trajectory, we can not assume any additional regularity. This fact generates a lot of complications for the propagation of compactness and regularity and the unique continuation. Yet, it can give some additional informations like the fact that the reachable set in any fixed time is open and the smallness assumption for local controllability is only necessary in some lower order norms than the energy norm.
\begin{figure}[!h]
$$\ecriture{\includegraphics[width=0.7\textwidth]{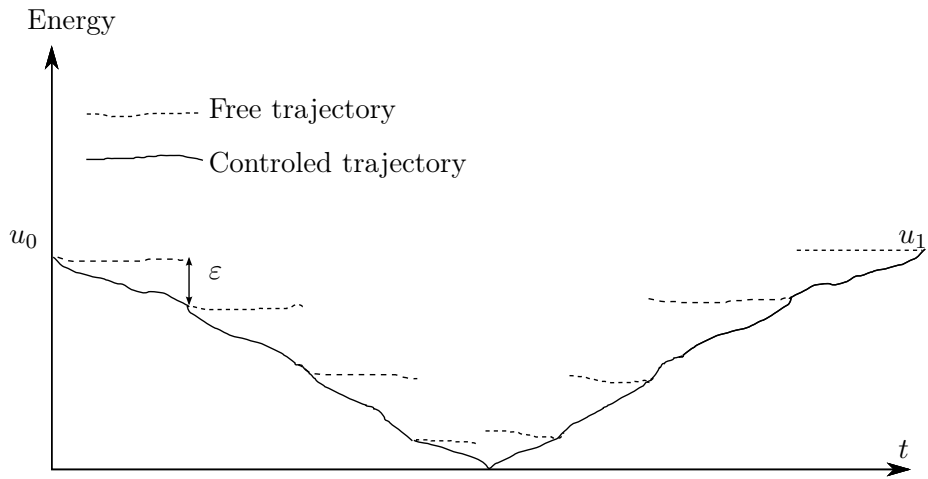}}
{\aat{0}{25}{Energy}\aat{48}{13}{$u_1$}\aat{-1}{13}{$u_0$}\aat{10}{20}{Free trajectory}\aat{10}{17}{Controled trajectory}\aat{48}{1}{$t$}\aat{10}{11}{$\varepsilon$}}$$ \caption{Global strategy by using succesive controls}\label{fig.strategy_successifs}
\end{figure}
\subsection{Unique continuation}
\label{subsectuniq}

In this subsection, we describe shortly the difficulties in proving the unique continuation property (\ref{eqnuniq}).

One of the main tools to prove unique continuation is Carleman estimates. In the case of Schrödinger equation, they allow to prove some local unique continuation property across some hypersurface:

\medskip
\emph{
Let $u$ be a solution of $i\partial_t u +\Delta u+V_1u+V_2\bar{u}=0$ such that locally $u(t,x)=0$ for $\varphi(x)\geq 0$, $t\in [0,T]$, near to a point $x_0$ where $\varphi(x_0)=0$,  then $u=0$ near $x_0$.
}
\medskip

 However, unlike the elliptic or parabolic case, $\varphi$ has to fulfill some geometric conditions, mainly $Hess ~\varphi > 0$. This condition is closely related to the strict pseudoconvexity condition which is necessary to get Carleman estimates (see Zuily \cite{Zuilybook}). However, since the Schrödinger operator is anisotropic, in our setting, pseudoconvexity needs only to be taken in the spatial variable, as proved more generally by Isakov \cite{Isakov} (see also \cite{LaurentNLSdim3} and \cite{LasieckaTrigH1} for some explicit computation). Given an open set $\omega$, the construction of the functions $\varphi$ that would allow to produce a global unique continuation result as (\ref{eqnuniq}) is not trivial. It is very restrictive with respect to the zone of control. One would desire some global result as: if $\omega$ satisfies the Geometric Control Condition, then the unique continuation (\ref{eqnuniq}) holds, but there is no such result for the moment so far. In fact, Miller \cite{Millerescape} gave some geometric examples of bounded open sets where the construction of pseudoconvex function for the wave operator (with boundary control) is impossible while Geometric Control Condition is fulfilled.

Some improvements on the geometric zone for unique continuation can be made by considering some weak Carleman estimate, where the function $\varphi$ fulfills only $Hess ~\varphi \geq 0$. For instance, in Mercado-Osses-Rosier \cite{CarlemanSchrodRosier} they prove unique continuation on rectangles where \\$\omega=\left\{x=(x_1,\ldots,x_n)\left| x_1\in [0,\varepsilon]\right. \right\}$ is a strip  (see also \cite{LaurentNLSdim3} by the author for the same result in some manifolds), by taking $\varphi(x)=x_1$.

Note that there exist some unique continuation results for partially analytic coefficients \cite{UniqZuilyRob}. This has been recently used by Joly and the author \cite{JolyLaurentNLW} for the stabilization of the nonlinear wave equation with only the Geometric Control Condition. An extension of this result for nonlinear Schrödinger equations would be very interesting. 

\appendix
\section{Appendix}
\bnp[Proof of Lemma \ref{lmpropageom}] We have to solve a transport equation with source term and with constraints of support.
Since $p$ is positive homogeneous of order $2$, there exists $q$ elliptic homogeneous of order $1$ such that $q^2=p$ (actually, the principal symbol of $q$ is $|\xi|_x$).\\
 We have $\left\{p,b\right\}=-\left\{b,q^2\right\}=- H_{b}q^2=- 2qH_{b}q= 2qH_{q}b$. Since $q$ is elliptic, we have to find $b$ and $r$ such that $H_{q}b= \frac{c}{2q} + \frac{r}{2q}$.

We are left with the following problem:\\
 Let $q$ real homogeneous elliptic of order $1$, $\tilde{c}$ homogeous of order $-1$, supported in a small conic neighborhood $V_0$ of $\rho_0$. We have to find $b$ and $\tilde{r}$ of order $-1$, with $\tilde{r}$ supported in a neighborhood $V_1$ as in the Lemma such that 
$$H_{q}b =\tilde{c}+\tilde{r}. $$
We would like to apply the homogeneous Darboux theorem (see Theorem 21.1.9 of \cite{HormanderIII}). So we have to guarantee that $H_{q}$ and $\xi\cdot\frac{\partial}{\partial \xi}$ are linearly independent.

Indeed, if we denote $G(x)=g^{ij}(x)$ as the matrix of the cofficients of $\sigma(\Delta)=\sum g^{ij}\xi_i \xi_j={ ^t\xi G\xi}$, we have $\partial_{\xi}\sigma(\Delta)= G\xi $ which is not zero if $\xi\neq 0$ ($G$ is invertible). So $H_{q}$ has a component along $\frac{\partial}{\partial x}$ and it is therefore independent of $\xi\cdot\frac{\partial}{\partial \xi}$.

Therefore, there exists a local symplectic homogeneous transformation, centered in $\rho_0$ $$\Phi(x,\xi)=(y_1(x,\xi),...,y_n(x,\xi),\eta_1(x,\xi),...,\eta_n(x,\xi))$$
with $y_i$ homogeneous in $\xi$ of order $0$, $\eta_i$ homogeneous in $\xi$ of order $1$, $\eta_1(x,\xi)=q(x,\xi)$ and $y(\rho_0)=0$.

From now on, the functions on $T^*M$ are defined with the new coordinates $(y_i,\eta_i)$. We have $H_{q}=H_{\eta_1}=\frac{\partial}{\partial y_1}$. Since $\rho_1=\Gamma(t_0)$, in the new coordinates, it gives $\rho_1=\rho_0+t_0 \frac{\partial}{\partial y_1}$.\\
We can choose $\varepsilon$ small enough and $V_1$ conical neighborhood of $\rho_1=\Gamma(t_0)$, $0<t_0<\varepsilon$, such that $\left\{V_1+t\frac{\partial}{\partial y_1}, t\in [-2\varepsilon,\varepsilon]\right\}$ is included in the domain of the chart $\Phi$ . Select $\varepsilon_1$ with $0<\varepsilon_1<t_0 /2$ and a conical open set $O\subset\R^{2d-1}$ such that $\rho_1+]-\varepsilon_1,\varepsilon_1[\times O \subset V_1$. We choose next $V_0=]-\varepsilon_1,\varepsilon_1[\times O$.

For a $\tilde{c}$ supported in $V_0$, we define: 
\begin{eqnarray}
\label{formuleprolong}\widetilde{b}(y_1,\cdots,y_n,\eta_1,\cdots,\eta_n)= \int_{-\infty}^{y_1}\tilde{c}(t,y_2,\cdots,y_n,\eta_1,\cdots,\eta_n)~dt
\end{eqnarray}
Then, $\widetilde{b}$ is supported in $\left\{V_0+t\frac{\partial}{\partial y_1};t\in [0,+\infty [\right\}$.\\
Let $\Psi \in C^{\infty}(\R)$ so that $\Psi(t)=1$ for $t\leq t_0-\varepsilon_1$ and $\Psi(t)=0$ for $t\geq t_0+\varepsilon_1$.\\
Set $b(y,\eta)=\Psi(y_1)\widetilde{b}(y,\eta)$, as it was already defined on the domain of the chart $\Phi$.\\
We compute
\begin{eqnarray*}
H_{q}b(y,\eta)=\frac{\partial}{\partial y_1}b=\Psi(y_1)\tilde{c}(y,\eta)+\Psi'(y_1)\widetilde{b}(y,\eta).
\end{eqnarray*}
Since $\Psi(y_1)=1$ on $y_1\leq t_0-\varepsilon_1$, particularly on $V_0$, we have 
$$\Psi(y_1)\tilde{c}(y,\eta)=\tilde{c}(y,\eta)$$
Moreover, $\tilde{r}:= \Psi'(y_1)\widetilde{b}$ is supported in 
$$\left\{t_0-\varepsilon_1\leq y_1\leq t_0+\varepsilon_1\right\}\cap \left\{V_0+t\frac{\partial}{\partial y_1};t\in [0,+\infty [\right\}= \rho_1+]-\varepsilon_1,\varepsilon_1[\times O \subset V_1.$$
\begin{figure}[!h]
$$\ecriture{\includegraphics[width=0.5\textwidth]{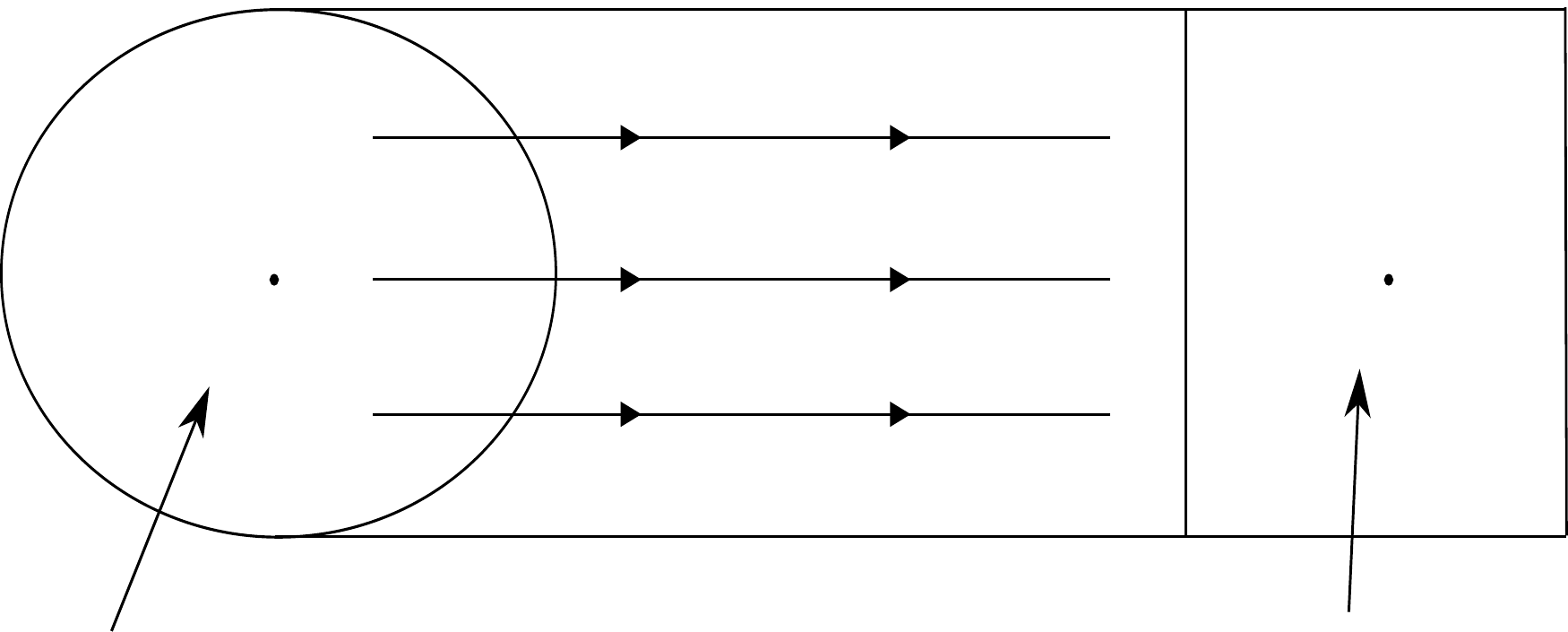}}
{\aat{0}{-1}{support of $\tilde{c}$}\aat{23}{11}{$H_q=\frac{\partial}{\partial y_1}$}\aat{9}{12}{$\rho_0$}\aat{43}{12}{$\rho_1$}\aat{37}{-1}{support of $\tilde{r}$}}$$ \caption{Propagation of information in phase space}\label{fig.propagGeom}
\end{figure}

Additionaly, if the symbol $\tilde{c}$ is homogeneous of order $-1$ in $\eta$, by the formula (\ref{formuleprolong}), $\tilde{r}$ and $\tilde{b}$ are also of order $-1$. By the homogeneous change of variable, it is also the case in the $\xi$ variable. We can check that all the symbols previously defined are compactly supported in the coordinate charts (up to dilation in the variable $\eta$) and they can be extended to $T^*M\setminus \left\{0\right\}$ in a smooth way.
\enp
For sake of completeness, we give a proof of the commutator estimate used in the proof of the Theorem \ref{ThmregHUM}, in the case of dimension 1. Note that this could be understood as a consequence that $D^r$ defined by (\ref{Dr}) is a pseudodifferential operator of order $r$.
\begin{lemme}
\label{lemmecommut}
Let $f$ denote the operator of multiplication by $f\in C^{\infty}(\Tu)$. Then, $[D^r,f]$ maps $H^s(\Tu)$ into $H^{s-r+1}(\Tu)$ for any $s,r\in \R$.
\end{lemme}
\begin{proof}Denote $\left| k\right|_{\wr}=\left|k\right|$ if $k\neq 0$ and $1$ otherwise. We also write $\sgn(0)=1$.
We have 
\bna
\widehat{D^r(fu)}(n)=\sgn(n)\left| n\right|_{\wr}^r\sum_k \widehat{f}(n-k)\widehat{u}(k)\\
\widehat{fD^ru}(n)=\sum_k \widehat{f}(n-k)\sgn(k)\left| k\right|_{\wr}^r\widehat{u}(k).
\ena
And then
\bna
\widehat{[D_r,f]u}(n)&=&\sum_k \widehat{f}(n-k)(\sgn(n)\left| n\right|_{\wr}^r-\sgn(k)\left| k\right|_{\wr}^r)\widehat{u}(k)\\
\left|\widehat{[D_r,f]u}(n)\right|&\leq & C\sum_k |\widehat{f}(n-k)||n-k|(\left| n\right|_{\wr}^{r-1}+\left| k\right|_{\wr}^{r-1})|\widehat{u}(k)|.
\ena
Using $\left| n\right|_{\wr}^{2\rho} \leq C\left| n-k\right|_{\wr}^{2|\rho|}\left| k\right|_{\wr}^{2\rho}$ for any $\rho \in \R$, we get
\bnan
\left\|[D_r,f]u\right\|^2_{H^{s-r+1}} & \leq & C\sum_{n}\left| n\right|_{\wr}^{2s}\left(\sum_{k} \left|\widehat{f}(n-k)(n-k)\right||\widehat{u}(k)|\right)^2\nonumber\\
&+&C\sum_{n}\left(\sum_{k}\left| n-k\right|_{\wr}^{|s-r+1|}\left| k\right|_{\wr}^{s} \left|\widehat{f}(n-k)(n-k)\right||\widehat{u}(k)|\right)^2\nonumber\\
&\leq &C\sum_{n}\left(\sum_{k}\left| n-k\right|_{\wr}^{|s|}\left| k\right|_{\wr}^{s} \left|\widehat{f}(n-k)(n-k)\right||\widehat{u}(k)|\right)^2 \label{terme1}\\
&+&C\sum_{n}\left(\sum_{k}\left|n-k\right|_{\wr}^{|s-r+1|}\left|k\right|_{\wr}^{s} \left|\widehat{f}(n-k)(n-k)\right||\widehat{u}(k)|\right)^2\label{terme2}.
\enan
We estimate (\ref{terme1}) using Cauchy-Schwarz inequality, and as well for (\ref{terme2}).
\bna
(\ref{terme1})&\leq &C\sum_{n}\left(\sum_{k}\left| n-k\right|_{\wr}^{|s|}|\widehat{f}(n-k)(n-k)|\right) \times\\
& &\left(\sum_{k}\left|n-k\right|_{\wr}^{|s|}|\widehat{f}(n-k)(n-k)|\left| k\right|_{\wr}^{2s} |\widehat{u}(k)|^2\right)\\
&\leq &C\left(\sum_{k}\left| k\right|_{\wr}^{|s|}|k\widehat{f}(k)|\right)^2\left(\sum_{k}\left| k\right|_{\wr}^{2s} |\widehat{u}(k)|^2\right)\\
&\leq & C_{f} \nor{u}{H^s}^2.
\ena \end{proof}

\vspace{10mm}

\noindent {\bf Acknowledgements:} I would like to thank the organizers of the summer school PASI-CIPPDE in Santiago de Chile for proposing me to give that lecture. I also thank Ivonne Rivas and Matthieu Léautaud for useful comments about the preliminary versions of these notes. The author was partially supported by the "Agence Nationale de la
Recherche" (ANR), Projet Blanc EMAQS number ANR-2011-BS01-017-01.
\bibliographystyle{plain} 
\bibliography{biblio}
\end{document}